\documentclass[12pt]{amsart}
\usepackage{amssymb,amsmath,amsfonts,latexsym}

\newcommand{\newsection}[1]{\setcounter{equation}{0} \section{#1}}
\newcommand{\vp}{\varphi}

\newcommand{\clb}{\mathcal{B}}

\newcommand{\clh}{\mathcal{H}}
\newcommand{\cli}{\mathcal{I}}
\newcommand{\clj}{\mathcal{J}}

\newcommand{\clq}{\mathcal{Q}}

\newcommand{\cls}{\mathcal{S}}

\newcommand{\clz}{\mathcal{Z}}

\newcommand{\D}{\mathbb{D}}

\newcommand{\T}{\mathbb{T}}
\newcommand{\Z}{\mathbb{Z}}

\newcommand{\raro}{\rightarrow}

\newtheorem{theorem}{Theorem}[section]
\newtheorem{lemma}[theorem]{Lemma}
\newtheorem{proposition}[theorem]{Proposition}
\newtheorem{corollary}[theorem]{Corollary}
\theoremstyle{definition}
\newtheorem{definition}[theorem]{Definition}

\numberwithin{equation}{section}

\begin{document}

\title{Products of two orthogonal projections}

\author[Bhattacharjee]{Jaydeep Bhattacharjee}
\address{Indian Statistical Institute, Statistics and Mathematics Unit, 8th Mile, Mysore Road, Bangalore, 560059,
India}
\email{rs\_math2102@isibang.ac.in}

\author[Sarkar]{Jaydeb Sarkar}
\address{Indian Statistical Institute, Statistics and Mathematics Unit, 8th Mile, Mysore Road, Bangalore, 560059,
India}
\email{jay@isibang.ac.in, jaydeb@gmail.com}


	
\begin{abstract}
We study operators that are products of two orthogonal projections. Our results complement some of the classical results of Crimmins and von Neumann. Particular emphasis has been given to projections associated with inner functions defined on the polydisc.
\end{abstract}
	
\maketitle

\tableofcontents

\newsection{Introduction }\label{sec: intro}

Elements from an algebra that are finite products of special elements from the same algebra are of general significance. One of the more specific cases would be the ring of all square matrices or the ring of all bounded linear operators acting on a Hilbert space $\clh$, which we denote by $\clb(\clh)$. The present investigation focuses on operators that are products of pairs of projections on Hilbert spaces. Here, all Hilbert spaces are assumed to be separable and defined over $\mathbb{C}$, and all projections are orthogonal projections. Therefore, $P \in \clb(\clh)$ is a \textit{projection} if
\[
P = P^* = P^2.
\]
Our aim is to study operators $T \in \clb(\clh)$ such that
\[
T = P_1 P_2,
\]
for some projections $P_1$ and $P_2$ in $\clb(\clh)$. Operators that admit the above factorizations have been examined in multiple contexts. Aronszajn, Browder, Dixmier , Kakutani, and Wiener are among the many more eminent names who have contributed to illuminate this subject (see \cite{NS, Wu} for a thorough historical description). Two additional notable contributions relevant to our work are von Neumann's formula for iterated products of projections and Crimmins' analysis of products of two projections. Let us first revisit two characterizations from the list provided by Crimmins \cite{CM, RW, Wu}: An operator $T \in \clb(\clh)$ is the product of two projections if and only if
\[
T^2 = T T^* T,
\]
if and only if
\[
T = P_{\overline{\text{ran}}T} P_{\overline{\text{ran}}T^*}.
\]
Given a closed subspace $\cls$ of a Hilbert space $\clh$, we denote the orthogonal projection of $\clh$ onto $\cls$ by $P_\cls$. The preceding factorization of $T$ is referred to as the \textit{canonical factorization} and will be crucial in the subsequent discussion. We present somewhat independent proofs of the above equivalence properties. Additionally, we offer the following new characterization: $T \in \clb(\clh)$ is a product of two projections if and only if
\[
T T^* = T P_{\overline{\text{ran}}T^*}.
\]

Before turning to von Neumann's perspective, we pause for Nagy-Foias and Langer's viewpoint on contractive operators acting on Hilbert spaces. This is relevant because, if $T \in \clb(\clh)$ is a product of two projections, then $T$ is necessarily a contraction (that is, $\|T h\| \leq \|h\|$ for all $h \in \clh$). Consequently, the Sz.-Nagy and Foias theory of contractions \cite{NF Book} applies to operators that are products of two projections. This paper presents some new perspectives on the structure of contractions that are products of two projections. Given a contraction $T \in \clb(\clh)$, the celebrated \textit{canonical decompositions} of $T$ is the orthogonal decomposition of closed subspaces
\[
\clh = \clh_u \oplus \clh_{cnu},
\]
where both $\clh_u$ and $\clh_{cnu}$ reduce $T$, and $T|_{\clh_u}$ is unitary and $T|_{\clh_{cnu}}$ is completely non-unitary (c.n.u., for short); that is, $T|_{\clh_{cnu}}$ on $\clh_{cnu}$ has no nontrivial unitary summand. Moreover, the unitary summand $\clh_u$ is given by
\[
\clh_u = \{h \in \clh: \|T^m h\| = \|T^{*m} h\| = \|h\|, m \in \mathbb{N}\}.
\]
This decomposition is due to Nagy-Foias and Langer \cite{Langler, N-F 60}. The spaces $\clh_u$ and $\clh_{cnu}$ are designated as the \textit{unitary part} and the \textit{cnu} part of $T$, respectively.

Now we assume that $T = P_1P_2$ for some projections $P_1$ and $P_2$ in $\clb(\clh)$. Theorems \ref{thm: cnu ker T-I} and \ref{thm: cnu} give a concrete description of the canonical decomposition of such operators: The unitary and cnu parts of $T$ are given by
\[
\clh_u = \ker (I - P_1P_2) = \text{ran} P_1 \cap \text{ran} P_2,
\]
and
\[
\clh_{cnu} = \ker T \bigvee \ker T^*,
\]
respectively, where $\bigvee$ denotes the span closure of subspaces. The above results also complement the classical von Neumann's alternating orthogonal projection formula. In fact, if $P_1$ and $P_2$ are projections, then the von Neumann's alternating orthogonal projection formula tells us that
\[
\text{SOT-} \lim_{m\raro \infty} T^m = P_{\text{ran} P_1 \cap \text{ran} P_2}.
\]
From this perspective and in view of our results outlined above, we further assert that (see Theorem \ref{thm: vn 1} for more details)
\[
\text{SOT-} \lim_{m\raro \infty} T^m = P_{\clh_u},
\]
where $\clh_u$ is the unitary part of the canonical decomposition of the contraction $T = P_1 P_2$. In the context of canonocial decomposition, we additionally include the following asymptotic properties of the cnu part of $T$ (see Theorem \ref{thm: vN and proj}):
\[
\text{SOT}-\lim_{m \raro \infty} (T|_{\clh_{cnu}})^{*m} = \text{SOT}-\lim_{m \raro \infty} (T|_{\clh_{cnu}})^{m} = 0.
\]
Following standard notations \cite{NF Book}, we write the above asymptotic properties simply as: 
\[
T^*|_{\clh_{cnu}}, T|_{\clh_{cnu}} \in C_{\cdot 0}.
\]

Now we look into a more concrete infinite dimensional Hilbert space, $H^2(\D^n)$, the Hardy space of square summable analytic functions defined on the polydisc $\D^n$. The commutative Banach algebra of all bounded analytic functions on $\D^n$ is denoted by $H^\infty(\D^n)$:
\[
H^\infty(\D^n) = \{\vp \in \text{Hol}(\D^n): \|\vp\|_\infty := \sup_{z \in \D^n} |\vp(z)| < \infty\}.
\]
For each $\vp \in H^\infty(\D^n)$, the analytic Toeplitz operator $T_\vp$ on $H^2(\D^n)$ is defined by
\[
T_\vp f = \vp f,
\]
for all $f \in H^2(\D^n)$. One knows that $T_\vp \in \clb(H^2(\D^n))$ for all $\vp \in H^\infty(\D^n)$. The function $\vp \in H^\infty(\D^n)$ is said to be \textit{inner} if $|\vp| = 1$ a.e. on $\T^n$ (the distinguished boundary of $\D^n$) in the sense of radial limits. It is known that $\vp$ is inner if and only if $T_\vp$ is an isometry on $H^2(\D^n)$. In this case, we have
\[
T_\vp T_\vp^* = P_{\vp H^2(\D^n)}.
\]
A projection $P \in \clb(H^2(\D^n))$ is said to be an \textit{inner projection} \cite{Debnath} if there exists an inner function $\vp \in H^\infty(\D^n)$ such that $P = T_\vp T_\vp^*$. Equivalently, we have
\begin{equation}\label{eqn: inn prof}
P = P_{\vp H^2(\D^n)}.
\end{equation}
Here our goal is to classify a class of operators that are products of two inner projections. Of course, our aim is to obtain an analytic answer to this problem. To achieve this, for each nonzero operator $T \in \clb(H^2(\D^n))$, we associate an inner function that acts as the least common multiple within a suitable class of inner functions. Our approach is as follows: Given a bounded linear operator $T$ on $H^2(\D^n)$, define the set
\[
\cli_T =\{\vp \in H^\infty(\D^n): \vp \text{ is inner, and }{\text{ran} T} \subseteq \vp H^2(\D^n)\}.
\]
Whenever $\cli_T$ admits a least common multiple (lcm), it is unique, and we denote it by $\vp_T \in H^\infty(\D^n)$. That is,
\[
\vp_T := \text{lcm} \cli_T.
\]
Under this notation, we similarly denote by
\[
\vp_{T^*} = \text{lcm} \cli_{T^*},
\]
whenever the lcm of $\cli_{T^*}$ exists. We also define
\[
\clj_T = \{\vp \in H^\infty(\D^n): \vp \text{ is inner and } \vp H^2(\D^n) \subseteq \ker T^*\}.
\]
Whenever it exists, we denote by $\psi_T$ the greatest common divisor (gcd) of $\clj_T$; that is,
\[
\psi_T:= \text{gcd} \clj_T.
\]

\begin{definition}
A nonzero operator $T \in \clb(H^2(\D^n))$ satisfies the lcm property (respectively, the gcd property) if both $\cli_T$ and $\cli_{T^*}$ (respectively, $\clj_T$ and $\clj_{T^*}$) contain a lcm (respectively, a gcd).
\end{definition}

In the one-variable setting, it follows from Beurling's theorem that every nonzero operator satisfies both the lcm and gcd properties (see Sections \ref{sec:inner proj} and \ref{sec: model proj}). In the spirit of canonical factorizations, Theorem \ref{thm: prod of inner} yields an analytic characterization of products of two inner projections: Let $T \in \clb(H^2(\D^n))$ be a nonzero operator satisfying the lcm property. Then $T$ is a product of two inner projections if and only if
\[
T = P_{\vp_T H^2(\D^n)} P_{\vp_{T^*} H^2(\D^n)}.
\]
Moreover, we observe in Corollary \ref{cor: not prod of inner} that if $T \in \clb(H^2(\D^n))$ is a cnu contraction, then $T$ cannot be expressed as a product of two inner projections. 

We also obtain similar results in the context of model spaces. A \textit{model space} $\clq_\vp$ is defined as the closed subspace obtained by quotienting $H^2(\D^n)$ by an inner function $\vp \in H^\infty(\D^n)$ in the sense that
\[
\clq_\vp = H^2(\D^n) \ominus \vp H^2(\D^n).
\]
We refer to $P_{\clq_\vp}$ the projection onto the model space $\clq_\vp$ as a \textit{model projection}. We prove the following (see Theorem \ref{thm: model proj}): Let $T \in \clb(H^2(\D^n))$ be a nonzero operator satisfying the gcd property. Then $T$ is a product of two model projections if and only if
\[
T = P_{\clq_{\psi_T}} P_{\clq_{\psi_{T^*}}}.
\]

We note that, unlike $\cli_T$, the set $\clj_T$ has the potential to be empty. Therefore, part of the above result also guarantees that the set $\clj_T$ remains nonempty. The final section of this paper presents some examples illustrating some of our results.

The results of this paper can be divided into three parts.

\begin{enumerate}
\item Characterizations of operators that can be represented as products of two projections. Most of the results presented here are known; however, the proofs provided offer new insights that are useful for developing other results in the paper. This material is covered in Section \ref{sec: Crimmins}.
\item Applications of these characterizations to the classical Nagy–Foiaș and Langer orthogonal decompositions of contractions, as well as to von Neumann’s alternating orthogonal projection formula. These applications are discussed in Sections \ref{sec: can decom} and \ref{sec: von Neum}.
\item The study of more specialized classes of products of orthogonal projections, particularly those related to the Hardy space and to inner functions on the polydisc. These topics are developed in Sections \ref{sec:inner proj}, \ref{sec: model proj}, and \ref{sec: example}.
\end{enumerate}

\section{Crimmins' perspective}\label{sec: Crimmins}

In this section, we prove some results on operators that can be represented as products of two projections. We take Crimmins' point of view and dedicate a part of the section to reproving his results. However, we address this within a marginally wider framework, anticipating more realistic uses in the structure of operators that can be expressed as products of operators.

In the context of Crimmins' perspective, we remark that the equivalence of (1), (2), and (4) in Theorem \ref{cor: mixed classif of PP} of this section is attributed to Crimmins and is referred to in \cite{CM, RW, Wu}. However, we were unable to locate the original paper of Crimmins. It appears that Crimmins' result was first mentioned, without an explicit reference, by Radjavi and Williams in their 1969 paper \cite{RW}.

Let us start with a lemma that states that operators with projection factors on the left side are obligated to have a canonical choice for the left projection factors.

\begin{lemma}\label{lemma: PXP}
Let $T \in \clb(\clh)$. If $T = P X$ for some projection $P \in \clb(\clh)$ and linear operator $X \in \clb(\clh)$, then
\[
T = P_{\overline{\text{ran} T}} X.
\]
\end{lemma}
\begin{proof}
Since $T = PX$, by Douglas' lemma, it follows that $\overline{\text{ran} T} \subseteq \text{ran} P$. Then
\[
P P_{\overline{\text{ran} T}} = P_{\overline{\text{ran} T}} = P_{\overline{\text{ran} T}} P.
\]
Now, $T = P_{\overline{\text{ran} T}} T$, and hence $T = P_{\overline{\text{ran} T}} P X$. This implies $T = P_{\overline{\text{ran} T}} X$, which completes the proof of the lemma.
\end{proof}

As a result, operators admitting the left and right side projection factors can be represented as follows:

\begin{proposition}\label{prop: Crimmins 1}
Let $T \in \clb(\clh)$. If $T = P_1 X P_2$ for some projections $P_1, P_2 \in \clb(\clh)$ and linear operator $X \in \clb(\clh)$, then
\[
T = P_{\overline{\text{ran} T}} X  P_{\overline{\text{ran} T^*}}.
\]
\end{proposition}
\begin{proof}
By Lemma \ref{lemma: PXP}, we know that $T = P_{\overline{\text{ran} T}} X P_2$, and hence
\[
T^* = P_2 X^* P_{\overline{\text{ran} T}}.
\]
The proof can now be obtained by applying Lemma \ref{lemma: PXP} one more time to $T^*$ and the above factorization of $T^*$.
\end{proof}

In particular, if $X = I$, then we have the following classification of operators as products of two projections:

\begin{corollary}\label{cor: Crimmins 1}
Let $T \in \clb(\clh)$. Then $T$ is a product of two projection if and only if
\[
T = P_{\overline{\text{ran} T}} P_{\overline{\text{ran} T^*}}.
\]
\end{corollary}

Crimmins is credited with this result \cite[page 1595]{CM}. Here, we draw the conclusion from a rather broad perspective.

The following result is from \cite[Corollary 2]{Sebest}. The present proof, compared to \cite{Sebest}, is more elementary (the sufficient part) and is also presented for completeness. The result is clearly in line with Douglas' range inclusion theorem.

\begin{proposition}\label{prop: Sybest}
Let $T_1, T_2 \in \clb(\clh)$. Then there exists a projection $P \in \clb(\clh)$ such that
\[
T_1 = T_2 P,
\]
if and only if
\[
T_1 T_1^* = T_2  T_1^*.
\]
\end{proposition}
\begin{proof}
If $T_1 = T_2 P$ for some projection $P \in \clb(\clh)$, then
\[
T_1 T_1^* = T_2 P P T_2^* = T_2 P T_2^* = T_2 T_1^*.
\]
Conversely, assume that $T_1 T_1^* = T_2  T_1^*$, then $T_1f = T_2f$ for all $f \in \overline{\text{ran} T_1^*}$. Pick $f \in \clh$ and write
\[
f = P_{\ker T_1} f \oplus P_{\overline{\text{ran} T_1^*}} f.
\]
Then $T_1f = T_1 P_{\overline{\text{ran} T_1^*}} f$. Since $P_{\overline{\text{ran} T_1^*}} f \in \overline{\text{ran} T_1^*}$, we have
\[
T_1 f = T_2 P_{\overline{\text{ran} T_1^*}} f,
\]
which completes the proof by choosing the projection $P = P_{\overline{\text{ran} T_1^*}}$.
\end{proof}

In addition to its more elementary nature, the proof above also adds to our understanding, revealing that the projection $P$ is given by
\[
P = P_{\overline{\text{ran} T_1^*}},
\]
whenever $T_1= T_2 P$. The above proposition, though simple, is important for our purposes, as will become evident in a later section of the paper when we characterize operators that are products of two inner or model projections.

The subsequent corollary is yet another characterization of operators that are products of two projections.

\begin{corollary}\label{cor: prod of proj}
Let $T \in \clb(\clh)$. Then $T$ is a product of two projections if and only if
\[
T T^* = P_{\overline{\text{ran} T}} T^*.
\]
\end{corollary}
\begin{proof}
If $T$ is a product of two projections, then Corollary \ref{cor: Crimmins 1} implies $T = P_{\overline{\text{ran} T}} P_{\overline{\text{ran} T^*}}$. We now compute
\[
T T^* = P_{\overline{\text{ran} T}} P_{\overline{\text{ran} T^*}} P_{\overline{\text{ran} T}} = P_{\overline{\text{ran} T}} (P_{\overline{\text{ran} T^*}} P_{\overline{\text{ran} T}}) = P_{\overline{\text{ran} T}} T^*.
\]
For the converse direction, assume that $T T^* = P_{\overline{\text{ran} T^*}} T^*$. The conclusion now follows directly from the sufficient part of Proposition \ref{prop: Sybest}.
\end{proof}

Next, we assemble all the results obtained so far, along with a new one.

\begin{theorem}\label{cor: mixed classif of PP}
Let $T \in \clb(\clh)$. The following are equivalent:
\begin{enumerate}
\item $T$ is a product of two projections.
\item $T = P_{\overline{\text{ran}}T} P_{\overline{\text{ran}}T^*}$.
\item $T T^* = T P_{\overline{\text{ran} T}}$.
\item $T T^* T = T^2$.
\end{enumerate}
\end{theorem} 
\begin{proof}
The equivalences (1) $\Leftrightarrow$ (2) and (1) $\Leftrightarrow$ (3) are simply Corollary \ref{cor: Crimmins 1} and Corollary \ref{cor: prod of proj}, respectively. Next, assume (1), that is, $T$ is a product of two projections, say, $T = PQ$. Then
\[
T T^*T = PQ (QP)PQ = PQ PQ = T^2.
\]
This proves (4). To prove (4) $\Rightarrow$ (3), we assume that $T T^* T = T^2$, equivalently, $T^* T T^* = T^{*2}$. Then
\[
0 = T^*(T T^* - T^*) = T^*(P_{\overline{\text{ran} T}} + (I - P_{\overline{\text{ran} T}}))(T T^* - T^*) = T^*P_{\overline{\text{ran} T}}(T T^* - T^*),
\]
implies
\[
(T T^* - T P_{\overline{\text{ran} T}}) T = 0.
\]
This says $T T^* = T P_{\overline{\text{ran} T}}$ on $\overline{\text{ran} T}$. Since the identity holds trivially on $(\overline{\text{ran} T})^\perp = \ker T^*$, we conclude that $T T^* = T P_{\overline{\text{ran} T}}$. This completes the proof of the theorem.
\end{proof}

The equivalence of (1) and (4) above is due to Crimmins \cite[page 1595]{CM}. The present proof is once again different. In this context, we also refer the reader to \cite[Theorem 3.1]{CM}. The same paper contains results (cf. \cite[Corollary 3.8]{CM}) related to parametrizations and the uniqueness of pairs of projections whose products yield a given operator.

 

\section{Nagy-Foias and Langer’s perspective}\label{sec: can decom}

Note that if a bounded linear operator acting on a Hilbert space is a product of two projections, then it is necessarily a contraction. Consequently, the theory of contractions, widely recognized as being extremely rich, applies in this situation. This is a vast domain, and a thorough application of the theory of contractions to this particular class of operators and vice versa may not be entirely transparent. This requires a more in-depth investigation, which ought to be carried out. However, in this section, we will address a fundamental characteristic of the product of projections via the structure of contractions. To do that, we revisit the classic Nagy-Foias and Langer orthogonal decompositions of contractions \cite{Langler, N-F 60}.

Let $T \in \clb(\clh)$ be a contraction. Then there is a unique orthogonal decomposition $\clh = \clh_u \oplus \clh_{cnu}$ such that both $\clh_u$ and $\clh_{cnu}$ reduce $T$, and $T|_{\clh_u}$ is unitary whereas $T|_{\clh_{cnu}}$ is completely non-unitary (cnu in short). This amounts to saying that $T|_{\clh_{cnu}}$ does not have unitary summand (see Section \ref{sec: intro}). As a result, we get the diagonal decomposition:
\[
T = \begin{bmatrix}
T|_{\clh_u} & 0
\\
0 & T|_{\clh_{cnu}}
\end{bmatrix},
\]
on $\clh = \clh_u \oplus \clh_{cnu}$. Moreover, we have the representation of the unitary part $\clh_u$ as \cite[Chapt. 1, Theorem 3.2]{NF Book}
\[
\clh_u = \{h \in \clh: \|T^m h\| = \|T^{*m} h\| = \|h\|, m \in \mathbb{N}\}.
\]
This decomposition is commonly known as the \textit{Nagy-Foia\c{s} and Langer} or the \textit{canonical decomposition} of contractions and stands as one of the most basic structures of contractions. In the following, we connect this with our theory of the product of two projections. For simplicity of notation, given projections $P_1$ and $P_2$, define the closed subspace $\cls_{P_1 P_2}$ as
\begin{equation}\label{eqn: S P}
\cls_{P_1 P_2} = \text{ran} P_1 \cap \text{ran} P_2.
\end{equation}

\begin{theorem}\label{thm: cnu ker T-I}
Let $P_1$ and $P_2$ be projections defined on some Hilbert space $\clh$. Then
\[
\clh_u = \ker (P_1 P_2 - I) = \cls_{P_1 P_2},
\]
where $\clh_u$ is the unitary part of the contraction $P_1 P_2$.
\end{theorem}
\begin{proof}
First, we claim that
\begin{equation}\label{eqn: range of P}
\cls_{P_1 P_2} = \{f \in \clh: \|P_1 P_2f\| = \|f\|\}.
\end{equation}
Let $f \in \clh$, and assume that $\|P_1 P_2f\| = \|f\|$. Then
\[
\|f\| = \|P_1 P_2 f\| \leq \|P_2f\| \leq \|f\|.
\]
Subsequently, each of the preceding inequities attains equality. Therefore
\[
\|P_2f\| = \|P_1 P_2 f\| = \|f\|.
\]
Since $P_2$ is a projection, $\|P_2f\| = \|f\|$ implies that $f \in \text{ran} P_2$, or equivalently
\[
P_2 f= f.
\]
Again, $\|P_2f\| = \|P_1 P_2 f\|$ implies that $P_2 f \in \text{ran} P_1$. Therefore, $P_2 f = P_1 P_2 f$ and hence
\[
P_1 f= f.
\]
This proves that $f \in \text{ran} P_1 \cap \text{ran} P_2$. On the other hand, if $f \in \text{ran} P_1 \cap \text{ran} P_2$, then
\[
P_1 P_2f = f,
\]
and evidently $\|P_1 P_2f\| = \|f\|$. This completes the proof of the identity \eqref{eqn: range of P}. Now we turn to the proof of the main body of the theorem. Let $f \in \clh_u$. In particular, $\|P_1 P_2 f\| = \|f\|$, and then, \eqref{eqn: range of P} implies
\[
f \in \text{ran} P_1 \cap \text{ran} P_2.
\]
Therefore, $P_1P_2 f = f$, and we conclude that $f \in \ker (P_1 P_2 - I)$. This proves that $\clh_u \subseteq \ker (P_1 P_2 - I)$ and $\clh_u \subseteq \cls_{P_1 P_2}$. For the reverse inclusion, pick $g \in \ker (P_1 P_2 - I)$. Then $\|P_1 P_2 g\| = \|g\|$ and \eqref{eqn: range of P} together imply that
\[
g \in \text{ran} P_1 \cap \text{ran} P_2,
\]
and hence $P_2 P_1 g = g$. It follows that
\[
P_1 P_2 g = (P_1 P_2)^* g = g,
\]
and hence $g \in \clh_u$. This proves that $\ker (P_1 P_2 - I) \subseteq \cls_{P_1 P_2}$ and $\cls_{P_1 P_2} \subseteq \clh_u$. Therefore, $\clh_u = \ker (P_1 P_2 - I) = \cls_{P_1 P_2}$, which completes the proof of the theorem.
\end{proof}

In particular, the cnu part of $P_1 P_2$ is given by
\[
\clh_{cnu} = (\ker (P_1 P_2 - I))^\perp.
\]
However, some more can be said:

\begin{theorem}\label{thm: cnu}
If $T \in \clb(\clh)$ is a product of two projections, then
\[
\clh_{cnu} = \ker T \bigvee \ker T^*.
\]
\end{theorem}
\begin{proof}
By Corollary \ref{cor: Crimmins 1}, we know that
\[
T = P_{\overline{\text{ran} T}} P_{\overline{\text{ran} T^*}}.
\]
By Theorem \ref{thm: cnu ker T-I}, we know that
\[
\clh_u = \overline{\text{ran} T} \cap \overline{\text{ran} T^*}.
\]
Since $\ker T = (\overline{\text{ran} T^*})^\perp$, we have
\[
\clh_{cnu} = \clh_u^\perp = (\overline{\text{ran} T} \cap \overline{\text{ran} T^*})^\perp = \ker T \bigvee \ker T^*,
\]
which completes the proof of the theorem.
\end{proof}

Given this and Theorem \ref{thm: cnu ker T-I}, the canonical decompositions for contractions that are products of two projections are now concrete.

We end this section with an application of Theorem \ref{thm: cnu ker T-I} to a concrete situation, like inner projections defined on the polydisc (see the definition in \eqref{eqn: inn prof}). Let $T \in \clb(H^2(\D^n))$ be a nonzero operator. Suppose $T$ is a product of two inner projections, that is, there exist two inner functions $\vp$ and $\psi$ in $H^\infty(\D^n)$ such that
\[
T = P_{\vp H^2(\D^n)} P_{\psi H^2(\D^n)}.
\]
Theorem \ref{thm: cnu ker T-I} yields
\begin{equation}\label{cor: cnu=0}
\ker (T-I) = H^2(\D^n)_u = \vp H^2(\D^n) \cap \psi H^2(\D^n).
\end{equation}

On operators that cannot be expressed as a product of two inner projections, we have the following result:

\begin{corollary}\label{cor: not prod of inner}
If $T \in \clb(H^2(\D^n))$ is a cnu contraction, then $T$ cannot be expressed as a product of two inner projections. 
\end{corollary}
\begin{proof}
Suppose, if possible, that $T$ is a product of two inner projections. Since $T$ is cnu, we have that $H^2(\D^n)_u = \{0\}$ \cite{NF Book}. By \eqref{cor: cnu=0}, we conclude
\[
\{0\} = \vp H^2(\D^n) \cap \psi H^2(\D^n),
\]
which is a clear contradiction as $\vp \psi \in H^\infty(\D^n)$ is an inner function and
\[
\vp \psi \in \vp H^2(\D^n) \cap \psi H^2(\D^n).
\]
This completes the proof of the corollary.
\end{proof}

In a more general way, operators with unitary parts cannot be expressed as products of inner projections. 

\section{von Neumann's perspectives}\label{sec: von Neum}

When discussing products of two projections, one immediately draws a connection with the classical von Neumann's alternating orthogonal projections formula \cite{von}. This formula relates to the limit of powers of products of two projections. We begin with recalling a notation. Given projections $P_1$ and $P_2$ defined on some Hilbert space $\clh$, the closed subspace $\cls_{P_1 P_2}$ is the common range space defined by (see \eqref{eqn: S P})
\[
\cls_{P_1 P_2} = \text{ran} P_1 \cap \text{ran} P_2.
\]
The iterated product of projections then satisfies the following property, which was proved by von Neumann in 1933 \cite{von}:
\begin{equation}\label{eqn: vN limit}
\text{SOT-} \lim_{m\raro \infty} (P_1 P_2)^m = P_{\cls_{P_1 P_2}},
\end{equation}
where SOT refers to the strong operator topology. This result, known as von Neumann's alternating orthogonal projections formula, has historically been associated with notable mathematicians, including Aronszajn \cite{Aro}, Browder \cite{FB}, Kakutani \cite{Kak}, Nakano \cite{Nak}, and Wiener \cite{NW}, among others. We direct the reader to \cite{NS} for a comprehensive account of the development and a proof. There is more to say in this line, and it will be crucial to our understanding. In fact, we further have (see the proof of Theorem 1 in \cite{NS})
\[
\cls_{P_1 P_2} = \ker (I - P_2 P_1).
\]
Now we bring in Theorem \ref{thm: cnu ker T-I}, which tells us, in addition, that
\[
\clh_u = \ker (I - P_2 P_1) = \cls_{P_1 P_2}.
\]
This proves the following, which appears to be new information to von Neumann's alternating projection formula: 

\begin{theorem}\label{thm: vn 1}
Let $P_1, P_2 \in \clb(\clh)$ be two projections. Then
\[
\text{SOT-} \lim_{m\raro \infty} (P_1 P_2)^m = P_{\clh_u},
\]
where $\clh_u$ is the unitary part of the canonical decomposition of the contraction $P_1 P_2$.
\end{theorem}

Next, we introduce another significant class of contractions that serves a profound role in the theory of operators. Let $T \in \clb(\clh)$ be a contraction. If in addition
\[
\text{SOT}-\lim_{m \raro \infty} T^{*m} = 0,
\]
then we say that $T$ is a $C_{\cdot 0}$ contraction. If both $T$ and $T^*$ are in $C_{\cdot 0}$, then we simply write \cite[page 72]{NF Book}
\[
T \in C_{0 0}.
\]

Evidently, if $T \in \clb(\clh)$ is a $C_{0 0}$ contraction, then its unitary part is zero, that is, $\clh_u = \{0\}$. In the following, we focus on contractions that are products of pairs of projections. In fact, in the following result, we have added yet another new information to von Neumann's iterated products of projections:

\begin{theorem}\label{thm: vN and proj}
Let $T \in \clb(\clh)$ be a product of two projections. Then
\[
T|_{\clh_{cnu}} \in C_{0 0}.
\]
\end{theorem}
\begin{proof}
Assume without loss of generality that $T \in \clb(\clh)$ is a nonzero operator. Suppose $T = P_1 P_2$ for some projections $P_1$ and $P_2$ in $\clb(\clh)$. By part (4) of Theorem \ref{cor: mixed classif of PP} (which is due to Crimmins), we know that
\[
T T^* T = T^2.
\]
Consider the canonical decomposition of $T$ as $\clh = \clh_u \oplus \clh_{cnu}$. We know that $T_{cnu} := T|_{\clh_{cnu}}$ is a cnu contraction. Since $\clh_{cnu}$ reduces $T$, it follows that
\[
T = \begin{bmatrix}
T_u & 0 \\
0 & T_{cnu}
\end{bmatrix},
\]
where $T_u := T|_{\clh_u}$ is the unitary part of $T$. This immediately implies that
\[
T_{cnu} (T_{cnu})^* T_{cnu} = T_{cnu}^2.
\]
Applying Theorem \ref{cor: mixed classif of PP} again to $T_{cnu}$ on $\clh_{cnu}$, we conclude that it is the product of two projections. In particular, $T_{cnu}$ on $\clh_{cnu}$ is a contraction. We write the corresponding canonical decomposition of $T_{cnu}$ as
\[
\clh_{cnu} = \tilde{\clh}_u \oplus \tilde{\clh}_{cnu},
\]
where $\tilde{\clh}_u$ is the unitary part of $T_{cnu}$. Theorem \ref{thm: vn 1} now tells us that
\[
\text{SOT-} \lim_{m\raro \infty} T_{cnu}^m = P_{\tilde{\clh}_u}.
\]
However, $T_{cnu}$ is a cnu contraction, and hence, its unitary part is trivial, that is,
\[
\tilde{\clh}_u = \{0\},
\]
which says that $T|_{cnu} \in C_{\cdot 0}$. Finally, the fact that $T$ is a product of two projections implies that $T^*$ is also a product of two projections, thereby proving that $T|_{cnu}$ is in $C_{\cdot 0}$. This completes the proof of the theorem.
\end{proof}

Recall from Theorem \ref{thm: cnu}, given $T \in \clb(\clh)$, which is a product of two projections, we have
\[
\clh_{cnu} = \ker T \bigvee \ker T^*.
\]
Therefore, we have somewhat clear picture of the cnu part $T|_{\clh_{cnu}}$ of $T$. This observation shows significant potential for exploring the structure of this class of operators. However, we save this topic and direction for future work.

 
 

\newsection{Inner projections}\label{sec:inner proj}

This section presents a specific scenario of the product of projections where one might expect analytic solutions. The objective is to relate the concept of inner functions to projections, which was first explored under the name of inner projections in \cite{Debnath}. Recall that a projection $P \in \clb(H^2(\D^n))$ is called an \textit{inner projection} if there exists an inner function $\vp \in H^\infty(\D^n)$ such that
\[
P = P_{\vp H^2(\D^n)},
\]
or equivalently, $P = T_\vp T_\vp^*$.

Let $T \in \clb(H^2(\D^n))$. We already have learned from Corollary \ref{cor: Crimmins 1} that $T$ is a product of two projections if and only if $T = P_{\overline{\text{ran} T}} P_{\overline{\text{ran} T^*}}$. In this section, we additionally demand that the factors be inner projections, and in exchange we look for an analytic answer. Stated differently, our focus will be on identifying analytic interpretations of the projection factors. For this, we introduce the notion of lcm of bounded linear operators on $H^2(\D^n)$. Given a nonzero operator $T \in \clb(H^2(\D^n))$, define the set
\[
\cli_T =\{\vp \in H^\infty(\D^n): \vp \text{ is inner, and }{\text{ran} T} \subseteq \vp H^2(\D^n)\}.
\]
Since the constant function $1 \in H^\infty(\D^n)$ is also in $\cli_T$ (as, of course, $\text{ran} T \subseteq H^2(\D^n)$), it readily follows that
\[
\cli_T \neq \emptyset.
\]
We need to recall the notion of the lcm of inner functions. In what follows, our index set will always be nonempty, and the collection of functions will never be a singleton zero function. 

\begin{definition}\label{def: lcm}
Given a set of inner functions $\{\vp_\alpha\}_{\alpha \in \Lambda} \subseteq H^\infty(\D^n)$, an inner function $\vp \in H^\infty(\D^n)$ is said to be the \textit{least common multiple} (or \textit{lcm} in short) of $\{\vp_\alpha\}_{\alpha \in \Lambda}$ if
\begin{enumerate}
\item  $\vp_\alpha$ divides $\vp$ for all $\alpha \in \Lambda$, and
\item if $\vp_\alpha$ divides an inner function $\psi \in H^\infty(\D)$ for all $\alpha \in \Lambda$, then $\vp$ also divides $\psi$.
\end{enumerate}
\end{definition}

Clearly, the lcm function $\vp$, if it exists, is unique up to a unimodular constant. We simply denote it as
\[
\vp = \text{lcm}\{\vp_\alpha\}_{\alpha \in \Lambda}.
\]
Returning to the range of inner projections, we now want to investigate when the lcm exists. The result is a straightforward application of Beurling's theorem for the case when $n = 1$, but becomes complex when $n > 1$. We record the following well-known lemma and include a proof for completeness.

Recall that a nonzero operator $T \in \clb(H^2(\D^n))$ is said to satisfy the \textit{lcm property} if both $\cli_T$ and $\cli_{T^*}$ contains a lcm. We define the unique inner function (unique upto the multiplication by a unimodular scalar constant) $\vp_T \in H^\infty(\D)$ by
\[
\vp_T = \text{lcm} \cli_T.
\]
Given this notation, we also have the following:
\[
\vp_{T^*} = \text{lcm} \cli_{T^*},
\]
where $\cli_{T^*}$ is given by (as per the definition of $\cli_{T}$)
\[
\cli_{T^*} = \{\vp \in H^\infty(\D^n): \vp \text{ is inner, and }{\text{ran} T^*} \subseteq \vp H^2(\D^n)\}.
\]
Now that we have the pairs of inner functions $\vp_T$ and $\vp_{T^*}$ as constructed above, we are ready to provide an analytic description of operators that arise as products of inner projections. 

\begin{theorem}\label{thm: prod of inner}
Let $T \in \clb(H^2(\D^n))$ be a nonzero operator satisfying the lcm property. Then $T$ is a product of inner projection if and only if
\[
T = P_{\vp_T H^2(\D^n)} P_{\vp_{T^*} H^2(\D^n)}.
\]
\end{theorem}
\begin{proof}
We only need to prove the necessary part. Suppose there are inner functions $\vp_1$ and $\vp_2$ in $H^\infty(\D^n)$ such that $T = P_{\vp_1 H^2(\D^n)} P_{\vp_2 H^2(\D^n)}$. We recall, based on the construction of the set $\cli_{T^*}$, that
\[
\text{ran} T^* \subseteq \overline{\text{ran} T^*} \subseteq \vp_{T^*} H^2(\D^n),
\]
so that
\begin{equation}\label{eqn: proj 1}
P_{\overline{\text{ran} T^*}} = P_{\vp_{T^*} H^2(\D^n)} P_{\overline{\text{ran} T^*}}.
\end{equation}
In particular, $P_{\overline{\text{ran} T^*}}^* = P_{\overline{\text{ran} T^*}}$, or even the simple property of projections or set inclusions, implies that
\[
P_{\vp_{T^*} H^2(\D^n)} P_{\overline{\text{ran} T^*}} = P_{\overline{\text{ran} T^*}} P_{\vp_{T^*} H^2(\D^n)}.
\]
Similar to \eqref{eqn: proj 1}, or simply because $\overline{\text{ran} T^*} \subseteq \vp_{T^*} H^2(\D^n)$, we have
\[
T P_{\vp_{T^*} H^2(\D^n)} = T P_{\overline{\text{ran} T^*}} P_{\vp_{T^*} H^2(\D^n)},
\]
and then $P_{\vp_{T^*} H^2(\D^n)} P_{\overline{\text{ran} T^*}} = P_{\overline{\text{ran} T^*}} P_{\vp_{T^*} H^2(\D^n)}$ implies
\[
T P_{\vp_{T^*} H^2(\D^n)} = T P_{\vp_{T^*} H^2(\D^n)} P_{\overline{\text{ran} T^*}}.
\]
But, then \eqref{eqn: proj 1} yields
\[
T P_{\vp_{T^*} H^2(\D^n)} = T P_{\overline{\text{ran} T^*}}.
\]
Since $T|_{\ker T} = 0$, we immediately conclude that
\[
T = T P_{\vp_{T^*} H^2(\D^n)}.
\]
This identity is true for all nonzero $T \in \clb(H^2(\D^n))$. Consequently, if we apply this to $T^*$, then we see that
\[
T^* = T^* P_{\vp_{T} H^2(\D^n)}.
\]
Now, we first take the adjoint of this, and then apply the identity $T = T P_{\vp_{T^*} H^2(\D^n)}$ to deduce that
\[
T = P_{\vp_{T} H^2(\D^n)} T P_{\vp_{T^*} H^2(\D^n)}.
\]
Because we began with the factorization $T = P_{\vp_1 H^2(\D^n)} P_{\vp_2 H^2(\D^n)}$, for some inner functions $\vp_1, \vp_2 \in H^\infty(\D^n)$, we finally came to see that
\[
T =  P_{\vp_{T} H^2(\D^n)} (P_{\vp_1 H^2(\D^n)} P_{\vp_2 H^2(\D^n)}) P_{\vp_{T^*} H^2(\D^n)}.
\]
In addition, we are aware that $\vp_1$ divides $\vp_T$ and $\vp_2$ divides $\vp_{T^*}$ because of the property of lcm. This is equivalent to saying that
\[
P_{\vp_{T} H^2(\D^n)} P_{\vp_1 H^2(\D^n)} = P_{\vp_{T} H^2(\D^n)},
\]
and
\[
P_{\vp_2 H^2(\D^n)} P_{\vp_{T^*} H^2(\D^n)} = P_{\vp_{T^*} H^2(\D^n)}.
\]
As a result, $T = P_{\vp_T H^2(\D^n)} P_{\vp_{T^*} H^2(\D^n)}$, which concludes the proof of the theorem. 
\end{proof}

A comparison of Corollary \ref{cor: Crimmins 1} with the above shows that the latter corresponds to Crimmins’s variant of the product of inner projections.

One of the central questions addressed in \cite{Debnath} concerned the characterization of pairs of inner functions $\vp_1$  and $\vp_2$ in $H^\infty(\D^n)$ such that the inner projections $P_{\varphi_1 H^2(\D^n)}$ and $P_{\varphi_2 H^2(\D^n)}$ are commuting; equivalently, when is the product $P_{\varphi_1 H^2(\D^n)} P_{\varphi_1 H^2(\D^n)}$ itself a projection? The answer was given in \cite{Debnath} in terms of the inner functions themselves. Along this line, we ask the following: Given $T \in \mathcal{B}(H^2(\mathbb{D}^n))$ that is the product of two inner projections, how can we determine whether $T$ is also a projection? In this case, the answer is sought in terms of the operator $T$ itself. The following result provides our answer to this question.

\begin{theorem}\label{comm inner proj}
Let $T \in \clb(H^2(\D^n))$ be a product of two inner projections. Then $T$ is a projection if and only if $\ker T = \clq_{\varphi_{T}}$.
\end{theorem}
\begin{proof}	
Let $T = P_{\varphi_1 H^2(\D^n)} P_{\varphi_2 H^2(\D^n)}$, where $\varphi_1$ and $\varphi_2$ are inner functions in $H^\infty(\D^n)$. Suppose $T$ is a projection, that is,
\[
T = P_{\varphi_{1} H^2(\D^n) \cap \varphi_{2} H^2(\D^n)}.
\]
By \cite[Theorem 2.3]{Debnath}, or more precisely, by the discussion immediately following that theorem, there exists an inner function $\theta \in H^{\infty}(\D^n)$ such that
\[
\varphi_1 H^2(\D^n) \cap \varphi_2 H^2(\D^n) = \theta H^2(\D^n).
\]
In this case, it follows that
\[
T = P_{\theta H^2(\D^n)}.
\]
In particular, $\text{ran} T = \theta H^2(\D^n)$. Then the set $\cli_T$ becomes
\[
\cli_T = \{\varphi \in H^2(\D^n) : \varphi \text{ is inner and } \theta H^2(\D^n) \subseteq \varphi H^2(\D^n)\}.
\]
This implies that $\varphi_T = lcm \cli_T$ exists. By the definition of the lcm, it follows that there exists a scalar $\alpha \in \T$ such that
\[
\varphi_{T} = \alpha \theta,
\]
and hence
\[
\varphi_{T} H^2(\D^n) = \theta H^2(\D^n).
\]
We conclude that $T = P_{\theta H^2(\D^n)} = P_{\varphi_{T} H^2(\D^n)}$, and consequently
\[
\ker T = \clq_{\theta} = \clq_{\varphi_{T}}.
\]
Conversely, suppose that $\ker T = \clq_{\varphi_{T}}$. In other words, the lcm $\cli_{T}$ exists, and
\[
\overline{\text{ran}} T^* = \varphi_{T} H^2(\D^n).
\]
Thus $\varphi_{T^*} = lcm {\cli_{T^*}}$ also exists, where
\[
\cli_{T^*} = \{\varphi \in H^2(\D^n) : \varphi \text{ is inner and } ran T^* \subseteq \varphi H^2(\D^n)\},
\]
Thus, there exists a scalar $\alpha \in \T$ such that $\varphi_{T^*} = \alpha \varphi_{T}$, and hence
\[
\varphi_{T} H^2(\D^n) = \varphi_{T^*} H^2(\D^n).
\]
On the other hand, since $T$ satisfies the lcm property, Theorem $\ref{thm: prod of inner}$ implies that
\[
T = P_{\varphi_{T} H^2(\D^n)} P_{\varphi_{T^*} H^2(\D^n)},
\]
and hence $T = P_{\varphi_{T} H^2(\D^n)}$, which completes the proof.
\end{proof}

We remark that when $n=1$, every nonzero operator satisfies the lcm property. Thus, in the one-variable setting, the hypothesis of Theorem \ref{thm: prod of inner} is automatic. Although this follows (cf. \cite[page 21, Proposition 2.3]{Hari}) readily from the Beurling theorem, we include a detailed proof for the reader's convenience.

\begin{proposition}\label{lemma: lcm}
$\text{lcm} \cli_T$ exists for all nonzero $T \in \clb(H^2(\D))$.
\end{proposition}
\begin{proof}
Define
\[
\cls_T = \bigcap_{\vp \in \cli_T} \vp H^2(\D).
\]
Since
\[
\{0\} \neq \text{ran} T \subseteq \vp H^2(\D),
\]
for all $\vp \in \cli_T$, it follows that $\cls_T \neq \{0\}$. Since $\vp H^2(\D)$ is a closed $T_z$-invariant subspace of $H^2(\D)$ for all $\vp \in \cli_T$, it follows that $\cls_T$ is also a closed $T_z$-invariant subspace of $H^2(\D)$. According to the Beurling theorem, there exists an inner function $\vp \in H^\infty(\D)$ such that $\cls = \vp H^2(\D)$. It is now easy to see that $\vp = \text{lcm} \cli_T$.
\end{proof}

In particular, every nonzero operator in $H^2(\D)$ automatically satisfies the lcm property. The main obstacle to extending the above argument to several variables, $n > 1$, is that, given two inner functions $\vp_1$ and $\vp_2$ in $H^\infty(\D^n)$, it is not always the case that there exists an inner function $\theta \in H^\infty(\D^n)$ such that
\[
\varphi_1 H^2(\D^n) \cap \varphi_2 H^2(\D^n) = \theta H^2(\D^n).
\]

We conclude this section with the canonical decompositions of operators that result from products of inner projections, allowing us to once more link orthogonal decompositions with analytic objects. We focus on the single variable case. Let $T \in \clb(H^2(\D))$ be a product of inner projections. By \eqref{cor: cnu=0}, we know that $H^2(\D)_u = \vp_T H^2(\D) \cap \vp_{T^*} H^2(\D)$. Now, we know that
\[
\psi_T := gcd\{\vp_T, \vp_{T^*}\},
\]
is an inner function in $H^\infty(\D)$, and consequently, with respect to the Hilbert space decomposition $H^2(\D) = \psi_T H^2(\D) \oplus \clq_{\psi_T}$, we can write
\[
T = \begin{bmatrix}
I_{\psi_T H^2(\D)} & 0
\\
0 & T|_{\clq_{\psi_T}}
\end{bmatrix}.
\]
This decomposition bears some resemblance to the structure of matrices that result from the finite products of projections (see Wu \cite[Theorem 4.6]{Wu} and Oikhberg \cite{Timur}).

\section{Model projections}\label{sec: model proj}

Recall that the model space $\clq_\vp$ corresponding to an inner function $\vp \in H^\infty(\D^n)$ is the quotient space defined by
\[
\clq_\vp = H^2(\D^n)\ominus \vp H^2(\D^n).
\]
We refer to $P_{\clq_\vp}$ the projection onto the model space $\clq_\vp$ as a \textit{model projection}. This section seeks to understand operators that can be represented as products of two model projections. Similarly to inner projections, if a nonzero operator $T \in \clb(H^2(\D^n))$ can be represented as a product of two model projections, we can conclude that both the $\text{ran}T$ and $\text{ran} T^*$ are nonzero. For each nonzero $T \in \clb(H^2(\D^n))$, we define
\[
\clj_T = \{\vp \in H^\infty(\D^n): \vp \text{ is inner and } \vp H^2(\D^n) \subseteq \ker T^*\}.
\]
Clearly, the set $\clj_T$ defined above could potentially be empty (for instance, whenever $T$ is a co-isometry).

\begin{definition}\label{def: gcd}
Given a collection of inner functions $\{\vp_\alpha\}_{\alpha \in \Lambda} \subseteq H^\infty(\D^n)$, an inner function $\vp \in H^\infty(\D^n)$ is said to be the greatest common divisor (or gcd in short) of $\{\vp_\alpha\}_{\alpha \in \Lambda}$ if
\begin{enumerate}
\item $\vp$ divides $\vp_\alpha$ for all $\alpha \in \Lambda$, and
\item if an inner function $\psi \in H^\infty(\D)$ divides $\vp_\alpha$ for all $\alpha \in \Lambda$, then $\psi$ also divides $\vp$.
\end{enumerate}
\end{definition}

Recall that a nonzero operator $T \in \clb(H^2(\D^n))$ satisfies the \textit{gcd property} if both $\clj_{T}$ and $\clj_{T^*}$ contain a gcd. In the one-variable case, Beurling's theorem implies that every nonzero operator $T$ in $\clb(H^2(\D))$ for which both $\clj_{T}$ and $\clj_{T^*}$ are nonempty automatically satisfies the gcd property.

For a nonzero operator $T \in \clb(H^2(\D^n))$ satisfying the gcd property, we denote the unique (upto multiplication by a unimodular constant) gcd by
\[
\psi_T = gcd \clj_T. 
\]
Similarly, we have
\[
\psi_{T^*} = gcd \clj_{T^*}.
\]

\begin{theorem}\label{thm: model proj}
Let $T \in \clb(H^2(\D^n))$ be a nonzero operator satisfying the gcd property. Then $T$ is a product of two model projections if and only if
\[
T = P_{\clq_{\psi_T}} P_{\clq_{\psi_{T^*}}}.
\]
\end{theorem}
\begin{proof}
Suppose there are inner functions $\vp_1, \vp_2 \in H^\infty(\D^n)$ such that $T = P_{\clq_{\vp_1}} P_{\clq_{\psi_{\vp_2}}}$. Since $T$ satisfies the gcd property, we know that $\psi_T$ and $\psi_{T^*}$ are inner functions in $H^\infty(\D^n)$. Since $\psi_T \in \clj_T$ and is inner, it follows that
\[
\psi_T H^2(\D^n) \subseteq \ker T^*,
\]
and hence
\[
\overline{\text{ran} T} = (\ker T^*)^\perp \subseteq \clq_{\psi_T}.
\]
In particular, $\text{ran} T \subseteq \clq_{\psi_T}$. Now, we know in general that $\text{ran} (T P_{\clq_{T^*}}) \subseteq \text{ran} T$, which immediately implies that
\[
\text{ran} (T P_{\clq_{\psi_{T^*}}}) \subseteq \clq_{\psi_T},
\]
and hence we have the identity
\[
P_{\clq_{\psi_T}} T P_{\clq_{\psi_{T^*}}} = T P_{\clq_{\psi_{T^*}}}.
\]
Again, we know the general fact that $T = T P_{(\ker T)^\perp}$. The above identity implies
\[
T P_{\clq_{\psi_{T^*}}} = T P_{(\ker T)^\perp} P_{\clq_{\psi_{T^*}}}
\]
Next, as $\psi_{T^*} \in \clj_{T^*}$, we have $\psi_{T^*} H^2(\D^n) \subseteq \ker T$, equivalently, $(\ker T)^\perp \subseteq \clq_{\psi_{T^*}}$. This implies $P_{\clq_{\psi_{T^*}}} P_{(\ker T)^\perp} = P_{(\ker T)^\perp}$, and hence
\[
T P_{\clq_{\psi_{T^*}}} = T P_{(\ker T)^\perp} = T.
\]
This combined with $P_{\clq_{\psi_T}} T P_{\clq_{\psi_{T^*}}} = T P_{\clq_{\psi_{T^*}}}$ yields
\[
T = P_{\clq_{\psi_T}} T P_{\clq_{\psi_{T^*}}}.
\]
Now, as we know that $T$ is a product of projections $T = P_{\clq_{\vp_1}} P_{\clq_{\psi_{\vp_2}}}$, the above implies
\[
T = P_{\clq_{\psi_T}} P_{\clq_{\vp_1}} P_{\clq_{\psi_{\vp_2}}} P_{\clq_{\psi_{T^*}}}.
\]
Looking at the product of two projections $P_{\clq_{\psi_T}} P_{\clq_{\vp_1}}$, we observe that both $\psi_T$ and $\vp_1$ are in $\clj_T$, which implies by the definition of gcd that $\vp_1$ divides $\psi_T$; equivalently, $\clq_{\psi_T} \subseteq \clq_{\vp_1}$. This gives
\[
P_{\clq_{\psi_T}} P_{\clq_{\vp_1}} = P_{\clq_{\psi_T}},
\]
and also, similarly,
\[
P_{\clq_{\psi_{\vp_2}}} P_{\clq_{\psi_{T^*}}} = P_{\clq_{\psi_{T^*}}}.
\]
This completes the proof of the fact that $T = P_{\clq_{\psi_T}} P_{\clq_{\psi_{T^*}}}$.
\end{proof}

Similar to the case of inner projections, the above theorem is an analogue of Crimmins’s factorization result, as stated in Corollary \ref{cor: Crimmins 1}, but this time in the context of model projections.

At this time, we will make a general observation regarding kernels of product of two projections: When $T \in \clb(\clh)$ is expressed as a product of two projections, $T = P_1 P_2$, we obtain:
\begin{equation}\label{eqn: ker T}
\ker (P_1P_2) = [\text{ran}(I - P_1) \cap \text{ran} P_2] \oplus \text{ran}(I - P_2).
\end{equation}
Indeed, if $P_1 P_2 h = 0$ for some $h \in \clh$, then we write $h = h_r \oplus h_n \in \text{ran} P_2 \oplus \ker P_2$. Therefore
\[
0 = P_1 P_2 h = P_1 h_r,
\]
implies $(I - P_1) h_r = h_r \in \text{ran} P_2$, and consequently
\[
h = h_r \oplus h_n \in [\text{ran}(I - P_1) \cap \text{ran} P_2] \oplus \text{ran}(I - P_2),
\]
proving that $\ker T$ is contained in the right side subspace of \eqref{eqn: ker T}. The reverse set inclusion is straightforward.

This, when combined with Proposition \ref{prop: Sybest} in the context of inner projections, results in the following characterization of operators as products of two inner projections:

\begin{theorem}\label{thm: inner proj and ker}
Let $T \in \clb(H^2(\D^n))$ be a nonzero operator satisfying the lcm property. Then $T$ is a product of two inner projections if and only if
\[
T T^* = P_{\vp_T H^2(\D^n)} T^*,
\]
and
\[
\ker T = [\vp_{T^*} H^2(\D^n) \cap \clq_{\vp_{T}}] \oplus \clq_{\vp_{T^*}}.
\]
\end{theorem}
\begin{proof}
If $T$ is a product of two inner projections, then Theorem \ref{thm: prod of inner} implies
\[
T = P_{\vp_T H^2(\D^n)} P_{\vp_{T^*} H^2(\D^n)}.
\]
The representation of $\ker T$ then follows from \eqref{eqn: ker T}. For the other identity, we compute:
\[
T T^* = P_{\vp_T H^2(\D^n)} P_{\vp_{T^*} H^2(\D^n)} P_{\vp_T H^2(\D^n)} = P_{\vp_T H^2(\D^n)} T^*,
\]
We now turn to prove the sufficient part. By Proposition \ref{prop: Sybest} and $T T^* = P_{\vp_T H^2(\D^n)} T^*$, we know that
\[
T = P_{\vp_T H^2(\D^n)} P_{(\ker T)^\perp}.
\]
Define
\[
\tilde{T} = P_{\vp_T H^2(\D^n)} P_{\vp_{T^*} H^2(\D^n)}.
\]
We claim that $T = \tilde{T}$. The given assumption about $\ker T$ and the identity \eqref{eqn: ker T} yield
\[
\ker T = \ker \tilde{T}.
\]
In particular, $\clq_{\vp_{T^*}} \subseteq \ker T$, that is, $(\ker T)^\perp \subseteq \vp_{T^*} H^2(\D^n)$. This implies $P_{(\ker T)^\perp} = P_{\vp_{T^*} H^2(\D^n)} P_{(\ker T)^\perp}$, and hence
\[
T = P_{\vp_T H^2(\D^n)} P_{(\ker T)^\perp} = P_{\vp_T H^2(\D^n)} P_{\vp_{T^*} H^2(\D^n)} P_{(\ker T)^\perp} = P_{\vp_T H^2(\D^n)} P_{\vp_{T^*} H^2(\D^n)} = \tilde{T}.
\]
This completes the proof of the theorem.
\end{proof}

Finally, we again consider model projections. The proof in this case is similar to the previous one.

\begin{corollary}
Let $T \in \clb(H^2(\D^n))$ be a nonzero operator satisfying the gcd property. Then $T$ is a product of two model projections if and only if
and
\[
T T^* = P_{\clq_{\psi_T}} T^*,
\]
and
\[
\ker T = [\clq_{\psi_{T^*}} \cap \psi_{T} H^2(\D^n)] \oplus \psi_{T^*} H^2(\D^n).
\]
\end{corollary}

In the following section, we will use the identity \eqref{eqn: ker T} to concrete examples of projections. We will see that this simple observation brings some quick answers to questions related to the product of two projections.

We now point out that operators which are products of inner projections cannot be expressed as products of two model projections.

\begin{proposition}
Let $T \in \clb(H^2(\D^n))$ be a product of two inner projections. Then $T$ is not a product of two model projections.
\end{proposition}
\begin{proof}
Let $\varphi_1, \varphi_2$ be inner functions in $H^{\infty}(\D^n)$ such that
\[
T = P_{\varphi_1 H^2(\D^n)}P_{\varphi_2H^2(\D^n)}.
\]
By Theorem \ref{thm: cnu ker T-I}, we also have
\[
H^2(\D^n)_u = \varphi_1 H^2(\D^n) \cap \varphi_2 H^2(\D^n),
\]
where $H^2(\D^n)_u$ denotes the unitary part of the contraction $T$. In particular, $H^2(\D^n)_u$ is a nontrivial invariant subspace of $H^2(\D^n)$. Now, for a contradiction, suppose $T$ is a product of two model projections. Thus there are inner functions $\theta_1, \theta_2 \in H^{\infty}(\D^n)$ such that
\[
T = P_{\clq_{\theta_1}}P_{\clq_{\theta_2}}.
\]
Again, by Theorem $\ref{thm: cnu ker T-I}$, we have
\[
H^2(\D^n)_u = \clq_{\theta_1} \cap \clq_{\theta_2}.
\]
This means that $H^2(\D^n)_u$ is a joint $M_z = (M_{z_1}, \dots, M_{z_n})$ reducing subspace of $H^2(\D^n)$.
This forces $H^2(\D^n)_u = \{0\}$, which is a contradiction.
\end{proof}

We now revisit Theorem \ref{comm inner proj}. A key step in that result relies on the fact that the product of two commuting inner projections is again an inner projection. This, however, is not generally true: the product of two commuting model projections need not be a model projection whenever $ n > 1$. Nevertheless, in the spirit of Theorem \ref{comm inner proj}, we obtain the following result concerning model spaces.

\begin{theorem}\label{prod of model model}
Let $T \in \clb(H^2(\D^n))$ be a product of two model projections. Then $T$ is a model projection if and only if
\[
\ker T = \psi_{T} H^2(\D^n).
\]
\end{theorem}
\begin{proof}
Suppose $T = P_{\clq_\theta}$ for some inner function $\theta \in H^\infty(\D^n)$. Then $\ker T^* = \theta H^2(\D^n)$. As
\[
\clj_{T} = \{\varphi \in H^{\infty}(\D^n) : \varphi \text{ is inner and } \varphi H^2(\D^n) \subseteq \text{ ker }T^*\}
\]
it readily follows (from the definition of gcd) $\clj_T$ has a gcd. By our notation for the gcd of $\clj_T$, it follows $\psi_T H^2(\D^n) = \theta H^2(\D^n)$. Therefore,
\[
\text{ker }T = \text{ ker }T^* = \theta H^2(\D^n) = \psi_T H^2(\D^n).
\]
For the converse direction, assume that $\ker T = \psi_T H^2(\D^n)$. Again, as
\[
\clj_{T^*} = \{\varphi \in H^{\infty}(\D^n) : \varphi \text{ is inner and } \varphi H^2(\D^n) \subseteq \text{ ker }T\},
\]
it clearly follows that $\clj_{T^*}$ has a gcd. As by our notation, $\psi_{T^*} = $ gcd $\clj_{T^*}$, it follows that
\[
\psi_T H^2(\D^n) = \psi_{T^*} H^2(\D^n),
\]
and hence
\[
\clq_{\psi_T} = \clq_{\psi_{T^*}}.
\]
Here, both $\clj_{T}$, $\clj_{T^*}$ has a gcd, thus $T$ satisfies the gcd property. Since $T$ is a product of two model projections, by Theorem \ref{thm: model proj}, we know that
\[
T = P_{\clq_{\psi_T}} P_{\clq_{\psi_{T^*}}}.
\]
Therefore, we have $T = P_{\clq_{\psi_T}}$. This completes the proof.
\end{proof}

In closing this section, we remark that model spaces, and more generally, quotient modules over the polydisc, present both significant importance and challenges. We refer the reader to \cite{Monojit, Liaw, Debnath, Guo} and the references therein for related results on this theme.

\section{Blaschke products}\label{sec: example}

This section aims to provide concrete examples that illustrate some of the results obtained thus far. Given $\alpha \in \D$, the function $b_\alpha \in Aut(\D)$ defined by
\[
b_\alpha(z) = \frac{z - \alpha}{1 - \bar{\alpha} z} \qquad (z \in \D),
\]
is known as a \textit{Blaschke factor}. A finite Blaschke product is an inner function $\vp \in H^\infty(\D)$ such that
\[
\vp = \prod_{j=1}^m b_{\alpha_j},
\]
for some finite subset $\{\alpha_j\}_{j=1}^m \subset \D$. We set the zero set of $\vp$ as follows:
\[
\clz(\vp) = \{\alpha_j\}_{j=1}^m,
\]
counting the multiplicity. Blaschke products are important tools in Hilbert function spaces. In our context, we recall that finite Blaschke products provide finite codimensional invariant subspaces of $H^2(\D)$. More specifically, if $\clq_\vp$ is a model space for some inner function $\vp \in H^\infty(\D)$, then
\[
\text{dim} \clq_\vp < \infty,
\]
if and only if $\vp$ is a finite Blaschke product. Moreover, in this case, we have
\[
\text{dim} \clq_\vp = |\clz(\vp)|.
\]
The lemma that follows is elementary and well known. The cardinality of a set $A$ is denoted by the notation $|A|$.

\begin{lemma}\label{lem: Blaschk submodule}
Let $\vp_1$ and $\vp_2$ be two finite Blaschke products. Then
\[
\clq_{\vp_1} \cap \vp_2 H^2(\D) = \{0\},
\]
if and only if
\[
|\clz(\vp_1)| \leq |\clz(\vp_2)|.
\]
\end{lemma}
\begin{proof}
In the present setting, condition (4.1), as pointed out in \cite[Lemma 5.1]{Benhida}, is equivalent to the condition $\clq_{\vp_1} \cap \vp_2 H^2(\D) = \{0\}$. The conclusion then follows directly from \cite[Lemma 5.1]{Benhida}.
\end{proof}

In the following, we relate canonical factorizations of the products of inner projections with corresponding range spaces.

\begin{proposition}
Let $\vp_1, \vp_2\in H^\infty(\D)$ be finite Blaschke products. Assume that
\[
|\clz(\vp_1)| \neq |\clz(\vp_2)|.
\]
Then there does not exist any  $T \in \clb(H^2(\D))$ that is the product of two inner projections with the property that
\[
\overline{\text{ran} T} = \vp_1 H^2(\D) \text{ and } \overline{\text{ran} T^*} = \vp_2 H^2(\D).
\]
\end{proposition}
\begin{proof}
Let $T \in \clb(H^2(\D))$ be a product of two inner projections. Moreover, assume that $\overline{\text{ran}} T = \vp_1 H^2(\D)$ and $\overline{\text{ran} T^*} = \vp_2 H^2(\D)$. Since $\overline{\text{ran} T} = \vp_1 H^2(\D)$, by the definition of lcm (also see the construction of $\cli_T$), it follows that
\[
\vp_2 = \alpha \vp_T,
\]
for some $\alpha \in \T$. Similarly, we also have
\[
\vp_1 = \beta \vp_{T^*},
\]
for some $\beta \in \T$. In particular, we have
\[
\text{ran} T = \vp_2H^2(\D) = \vp_T H^2(\D),
\]
and
\[
\text{ran} T^* = \vp_1 H^2(\D) = \vp_{T^*} H^2(\D).
\]
The later identity implies $\ker T = \clq_{\vp_{T^*}}$. By \eqref{eqn: ker T}, we also know that
\[
\ker T = [\clq_{\vp_T} \cap \vp_{T^*} H^2(\D)] \oplus \clq_{\vp_{T^*}},
\]
and consequently
\[
\clq_{\vp_T} \cap \vp_{T^*} H^2(\D) = \{0\}.
\]
In other words, we have $\clq_{\vp_2} \cap \vp_1 H^2(\D) = \{0\}$, and hence, Lemma \ref{lem: Blaschk submodule} implies that $|\clz(\vp_2)| \leq |\clz(\vp_1)|$. Similarly, as $T^*$ is also a product of two inner projections, working as above, we find that $|\clz(\vp_1)| \leq |\clz(\vp_2)|$. This yields $|\clz(\vp_2)| = |\clz(\vp_1)|$ -- a contradiction, which completes the proof of the proposition.
\end{proof}

For operators that are products of two inner projections, the corresponding kernel spaces have a specific relationship with the set of all inner functions.

\begin{lemma}\label{lemma: ker T}
Let $T \in \clb(H^2(\D))$ be a product of two inner projections. Then there does not exist any nonconstant inner function $\vp \in H^\infty(\D)$ such that
\[
\ker T \subseteq \vp H^2(\D).
\]
\end{lemma}
\begin{proof}
Suppose $T$ is a product of two inner projections. We have, in particular, that $T = P_{\vp_T H^2(\D^n)} P_{\vp_{T^*} H^2(\D^n)}$. We again use the identity \eqref{eqn: ker T} and observe that
\[
\ker T = [\clq_{\vp_T} \cap \vp_{T^*} H^2(\D)] \oplus \clq_{\vp_{T^*}},
\]
In particular, $\clq_{\vp_T} \subseteq \ker T$, and hence, by assumption, $\clq_{\vp_T} \subseteq \vp H^2(\D)$. This can happen only if $\vp$ is a constant function. Indeed, $\clq_{\vp_T} \subseteq \vp H^2(\D)$ implies that $z^m \clq_{\vp_T} \subseteq \vp H^2(\D)$ for all $m \in \Z_+$, and then
\[
H^2(\D) = \vee_{m \in \Z_+} z^m \clq_{\vp_T} \subseteq \vp H^2(\D),
\]
proves the fact that $\vp$ is a constant function.
\end{proof}

As a consequence, we have the following result that says when an operator on $H^2(\D)$ can't be represented as a product of two inner projections. 

\begin{corollary}
Let $T \in \clb(H^2(\D))$ be a nonzero operator. If $\ker T$ is $T_z$-invariant, then $T$ is not a product of two inner projections.
\end{corollary}
\begin{proof}
Suppose $T$ is a product of two inner projections. If there exists an inner function $\vp \in H^\infty(\D)$ such that $\ker T = \vp H^2(\D)$, then Lemma \ref{lemma: ker T} will force that $\vp$ to be a constant function, implying that $T = 0$.
\end{proof}

For examples of bounded linear operators on $H^2(\D)$ that meet the criteria stated in the above result, we note that for every inner function $\theta \in H^\infty(\D)$, there exists a Hankel operator $H_\vp$, $\vp \in L^\infty(\T)$, such that \cite[page 15]{VP}
\[
\ker H_\vp = \theta H^2(\D).
\]
Clearly, there is now an abundance of examples of Hankel operators, as well as the operators referenced in the above corollary.

\vspace{0.2in}

\noindent\textbf{Acknowledgement:} We thank the referee for a careful reading of the manuscript and for pointing out some issues in the previous version. We are also grateful to Zeyou Zhu from Fudan University for identifying a serious error concerning Beurling-type invariant subspaces on the polydisc in the original manuscript, which led us to substantially revise the paper. The research of the second named author is supported in part by TARE (TAR/2022/000063) and MATRICS (ANRF/ARGM/2025/000130/MTR) by ANRF, Department of Science \& Technology (DST), Government of India.

\bibliographystyle{amsplain}

\end{document}